\documentclass[12pt]{amsart}
\usepackage{graphics,wrapfig, graphicx, mathrsfs,xcolor}

\usepackage{amsfonts,amsthm,amsmath,amssymb,amscd,mathrsfs,graphicx,xcolor,subfig,tikz}
\usepackage{graphics}
\usepackage{indentfirst}
\usepackage{cite}
\usepackage{bm, enumerate,xcolor}
\usepackage[dvips]{epsfig}
\usepackage{latexsym}
\usepackage{float}
\usepackage{mathtools}
\usepackage{amssymb,amscd,graphicx,xcolor,subfig,tikz}
\usepackage{graphics}
\usepackage{indentfirst}
\usepackage{bm, enumerate,xcolor}
\usepackage[dvips]{epsfig}
\usepackage{latexsym}
\usepackage{float}
\usepackage{amsmath}
\usepackage{amssymb}
\newtheorem{lemma}{Lemma}[section]
\newtheorem{theorem}{Theorem}[section]

\newtheorem{remark}{Remark}[section]

\numberwithin{equation}{section} \numberwithin{theorem}{section}
\numberwithin{example}{section} \numberwithin{remark}{section}
\numberwithin{figure}{section} \numberwithin{algorithm}{section}
\mathtoolsset{showonlyrefs}

\def\ba{\begin{array}}
\def\ea{\end{array}}
\def\bma{\left(\begin{matrix}}
\def\ema{\end{matrix}\right)}
\def\be{\begin{equation}}
\def\ee{\end{equation}}

\def\dfrac{\displaystyle\frac}

\title{On the rigidity of uniformly rotating vortex patch near the Rankine vortex}
\author{Yupei Huang}
\begin{document}
\maketitle
\begin{abstract}
    In this paper, we study the uniformly rotating vortex patch solutions for the 2D incompressible Euler equations. Specifically, we prove that if the patch solution is close to the Rankine vortex in a certain weak topology, it is either a Kirchhoff ellipse or a Rankine vortex.  
\end{abstract}
\section{introduction}
In this paper, we investigate certain rigidity aspects of the 2D incompressible Euler equations in $\mathbb{R}^2$. The 2D incompressible Euler equations are a widely used model in fluid dynamics to describe the motion of water. They can be written in the vorticity form as shown below:
\begin{equation}\label{euler}
    \begin{aligned}
    &\omega_{t}+u\cdot \nabla \omega=0,\\
    &u=\nabla^{\perp}\Delta^{-1}\omega,\\
    &\omega(t,x)=\omega_{0}(x),
    \end{aligned}
\end{equation}
here $\omega=\partial_2 u_1-\partial_1 u_2$ is the vorticity of the fluid velocity $u$ and $\nabla^{\perp}=(-\partial_2,\partial_1)$.
\par By the classical results of Yudovich (\cite{YUDOVICH19631407}), if we start with the initial data $\omega_0 \in L^{1}\cap L^{\infty}(\mathbb{R}^2)$, there is a global in-time solution $\omega \in L^{1}\cap L^{\infty}(\mathbb{R}^2)$. 
One fundamental question regarding the 2D incompressible Euler equations concerns the investigation of the long-term behaviors of generic solutions as time $t$ approaches infinity ($t \rightarrow +\infty$). In particular, within the context of a compact domain, there is a famous conjecture given by Shnirelman (\cite{shnirelman2013long}), see also \cite{drivas2023singularity} for further reference: 
\begin{itemize}
    \item redthe weak limit set of generic $L^{\infty}$ solutions is a compact set in $L^2$ topology that is invariant under the flow of Euler equations.
\end{itemize} \par There have recently been several interesting studies (\cite{modin2020casimir}, \cite{modin2022canonical}) aimed at understanding the structure of the compact set mentioned in the conjecture. However, a concrete characterization of this invariant set is still missing in the literature. What is certain is that this set should contain specific solutions for the 2D Euler equations, such as steady states, periodic solutions. To confirm a general solutions converge weakly to steady states or periodic solutions is a challenging problem, for global evolution of a general solution is notoriously hard to track. The state of the art is to investigate the global behavior of solution sufficiently close to certain steady states such as shear flows. In the work of \cite{ionescu2020inviscid}, inspired by \cite{bedrossian2015inviscid}, the authors studied solutions sufficiently close to the Couette flow in $\mathbb{T}\times [0,1]$ and demonstrated that their weak limits are still shear flows. Furthermore, in \cite{Masmoudi2020NonlinearID} established that the inviscid damping phenomenon, as detailed in \cite{bedrossian2015inviscid}, applies to solutions sufficiently close to general monotone shear flows in $\mathbb{T}\times [0,1]$. In the process of inviscid damping, phase mixing causes the transfer of information to smaller scales, resulting in the loss of enstrophy in the weak limit. The results mentioned above are all about the long-term behavior in the compact domain.\par
The study of the asymptotic behaviors of Euler equations on $\mathbb{R}^2$ presents significant challenges. For instance, one potential scenario for loss of enstrophy involves data escaping to spatial infinity.  The current understanding, to the author's knowledge, is limited to either nonlinear results near specific stationary solutions ( \cite{Ionescu2020NonlinearID}), or to linear results ( \cite{bedrossian2019vortex,ionescu2022linear}).\par
To better understand the global behaviors of 2D Euler solutions on $\mathbb{R}^2$, researchers have focused on compactly supported vortex patch solutions (\cite{drivas2023twisting, marchioro1994bounds}). Classical Yudovich theory (\cite{YUDOVICH19631407}) guarantees the global well-posedness of these patch solutions. However, even in this simpler context, characterizing their long-term behaviors remains an outstanding open problem. In this paper, we examine a special family of periodic patch solutions of the 2D Euler equations, referred to as V-states. These V-states have already attracted considerable attention in the fluid dynamics community, and we will detail the related results in the next subsection. 
\subsection{The V-states and main results}
The V-states constitute a family of special solutions to \eqref{euler}, taking the form:
\begin{equation}\label{Vstate}
\begin{aligned}
&\omega(t,x)=\chi_{\mathcal{D}_{t}}(x),\
&\mathcal{D}_{t}=x_0+e^{i\Omega t} \mathcal{D},
\end{aligned}
\end{equation}
where $\chi_{\mathcal{D}_{t}}$ is the characteristic function associated with the rotating set $\mathcal{D}_{t}$. In this context, $x_0$ represents the rotating center and $\Omega$ denotes the angular velocity. Explicit examples of V-states are relatively rare and include the Rankine vortex (where the patch is a disk), the annulus vortex patch, and the Kirchhoff ellipse (characterized by the vortex patch taking the shape of an ellipse).\par It was an open question in the nineteenth century to find examples of V-states other than those previously mentioned. The existence of \(m\)-fold symmetric V-states, also known as Kelvin \(m\) waves, was first suggested by Lord Kelvin in 1880 (see \cite{lamb1945dover}, p. 231). An argument for the existence of Kelvin \(m\) waves was provided in \cite{burbea1982motions}. This proof was later rigorously justified in \cite{hmidi2013boundary} using the Crandall-Rabinowitz bifurcation theorem. Since then, more examples of V-states have been constructed based on bifurcation theory (see \cite{castro2016uniformly,hmidi2016bifurcation, hmidi2016degenerate,hmidi2015rotating}). It should be noted that these V states all have smooth boundaries; As of now, no examples of compactly supported V-states with rough boundaries are known.\par
In addition to constructing examples of V-state solutions to \eqref{euler}, researchers also investigated their rigidity properties. In \cite{fraenkel2000introduction}, the moving plane method was employed to demonstrate that if \(\Omega=0\), and the patch is simply connected and smooth, then the V-state must be the Rankine vortex. Similarly, Hmidi in \cite{hmidi2015trivial} applied the same method to show that in the aforementioned setting, if \(\Omega<0\), the V-state is necessarily the Rankine vortex. More recently, in \cite{gomez2021symmetry}, the authors utilized the variational formulation of V-state solutions, and applied continuous Steiner symmetrization techniques to prove that for a simply connected compactly supported uniformly rotating patch solution with Lipshitz boundary, if \(\Omega\leq 0\) or \(\Omega\geq \frac{1}{2}\), then the patch must be a Rankine vortex. The analysis in the case where \(0<\Omega<\frac{1}{2}\) is considerably more subtle, as discussed in \cite{hassainia2020global}. As mentioned in the previous paragraph, in this regime, even the linearized operator near the Rankine vortex can be complex, potentially leading to nonradial \(m\)-fold symmetric V-states. In \cite{park2022quantitative}, the Liouville properties of the \(m\)-fold symmetric, simply-connected V-states were explored. Variational analysis was used to demonstrate that for a general family of \(m\)-fold symmetric V-states, which satisfy  reasonable symmetry constraints, the angular velocity \(\Omega\) must be greater than \(\frac{1}{2}-\frac{C}{m}\) for a universal positive constant \(C\).
\par
In this paper, we also investigate the rigidity of the V-state solutions where $0<\Omega<\frac{1}{2}$. We first show that if the V-state is close to Rankine vortices in a certain ``strong'' topology, then the V-state should be either a disk or an ellipse.\begin{theorem}\label{Gomez}
Let $\mathcal{D}$ be a V-state with angular velocity $\Omega$. Assume that the patch boundary $\partial \mathcal{D}$ can be described as a $C^{1,\frac{1}{2}}$ graph of its argument $\theta$: \begin{equation}
    R=R(\theta).
\end{equation}  
For sufficiently small $\delta$,
    if \begin{equation}\label{closeness}
        |R(\theta)-1|_{C^{1,\frac{1}{2}}(\mathbb{T})}\leq \delta,
    \end{equation}
    and \begin{equation}
        \left|\Omega-\frac{1}{4}\right|\leq \delta,
    \end{equation}
then $\mathcal{D}$ is a rotating disk or a rotating ellipse.
\end{theorem} The proof of Theorem \ref{Gomez} is greatly inspired by the bifurcation analysis in the remarkable works \cite{castro2016existence,hmidi2013boundary}. Once we have Theorem \ref{Gomez}, then by building some quantitative estimates on the V-states, we show that the rigidity persists even in a much weaker topology.
\begin{theorem}\label{main1}
    Let $\mathcal{D}$ be a V-state with a compact, simply connected support, rotating with angular velocity $\Omega$. Assume that the V-state satisfies the following three conditions:\begin{itemize}
        \item The center of rotation is close to the origin: $\displaystyle\left|\int_{\mathcal{D}}(x_1,x_2)dx_1dx_2\right|\leq \delta;$
        \item The angular velocity is close to $\frac{1}{4}$:
        $\displaystyle\left|\Omega-\frac{1}{4}\right|\leq \delta$;
    \item 
    The shape of the vortex patch is close to the unit disk $\mathbb{D}$: $\displaystyle \text{Area}(\mathbb{D}\Delta \mathcal{D})\leq \delta$. Here, $\Delta$ denotes the symmetric difference operation of two sets.  
\end{itemize}
If the parameter $\delta$ is small enough, then the V-state $\mathcal{D}$ is either a rotating disk or an ellipse.
\end{theorem}
We may also adapt our analysis to the more general case. We state the conclusion in the two remarks below without proof (since the proof follows verbatim from the proof of Theorem \ref{main1}).
\begin{remark}\label{main2}
     Let $\mathcal{D}$ be a V-state with  simply-connected and compact support, rotating with angular velocity $\Omega$. For a fixed $m$, assume the V-state $\mathcal{D}$ satisfies \begin{itemize}
         \item The center of rotation is close to the origin: $\displaystyle\left|\int_{\mathcal{D}}(x_1,x_2)dx_1dx_2\right|\leq \delta;$
        \item The angular velocity is close to $\frac{m-1}{2m}$:
        $\displaystyle\left|\Omega-\frac{m-1}{2m}\right|\leq \delta$;
    \item 
    The shape of the vortex patch is close to the unit disk $\mathbb{D}$: $\displaystyle \text{Area}(\mathbb{D}\Delta \mathcal{D})\leq \delta$. 
    \end{itemize}
    If the parameter $\delta$ is sufficiently small, then $\mathcal{D}$ is either a rotating Kelvin $m$-wave or a disk.
\end{remark}
\begin{remark}\label{main3}
     Let $\mathcal{D}$ be a V state with a compact and simple-connected support, rotating with angular velocity $\Omega$.  For a fixed $\Omega_0\notin \{\frac{1}{2}\}\cup\{\frac{m-1}{2m},m\in \mathbb{N}\}$, assume the V-state $\mathcal{D}$ satisfies \begin{itemize}
        \item The center of rotation is close to the origin: $\displaystyle\left|\int_{\mathcal{D}}(x_1,x_2)dx_1dx_2\right|\leq \delta;$
        \item The angular velocity is close to $\Omega_0$:
        $\displaystyle\left|\Omega-\Omega_0\right|\leq \delta$;
    \item 
    The shape of the vortex patch is close to the unit disk $\mathbb{D}$: $\displaystyle \text{Area}(\mathbb{D}\Delta \mathcal{D})\leq \delta$. 
\end{itemize}
If the parameter $\delta$ is sufficiently small, then $\mathcal{D}$ is a rotating disk.
\end{remark}
Here, we want to mention the quantity $\int_{\mathcal{D}}(x_1,x_2)dx_1dx_2$ in the following remark.
\begin{remark}\label{Center of vorticity}
    One may notice that $\int_{\mathcal{D}}(x_1,x_2)dx_1dx_2$ represents the location of the center of vorticity. In fact, in the case of a uniformly rotating patch, the rotating center is exactly the center of vorticity.
\end{remark}\begin{proof}[Proof of Remark \ref{Center of vorticity}]
Let $y$ be the center of rotation.
\begin{equation}
    \begin{aligned}
        &\quad\frac{2\pi}{\Omega}\int_{\mathcal{D}}zdzd\Bar{z}=\int_{0}^{\frac{2\pi}{\Omega}}\int_{\mathbb{R}^2}zw(z,0)dzd\Bar{z}dt\\&
        =\int_{0}^{\frac{2\pi}{\Omega}}\int_{\mathbb{R}^2}zw(z,t)dzd\Bar{z}dt=\int_{0}^{\frac{2\pi}{\Omega}}\int_{y+e^{i\Omega t}(\mathcal{D}-y)}zdzd\Bar{z}dt\\&=\int_{0}^{\frac{2\pi}{\Omega}} \int_{e^{i\Omega t}(\mathcal{D}-y)}(y+z)dzd\Bar{z}dt=\int_{0}^{\frac{2\pi}{\Omega}}\int_{D-y}(y+e^{-\Omega t}z)dzd\Bar{z}dt
    \end{aligned}
\end{equation}
Then by Fubini's Theorem, by first integrating in time, we have the following. \begin{equation}
    \frac{2\pi}{\Omega}\int_{\mathcal{D}}zdzd\Bar{z}=\frac{2\pi}{\Omega}|\mathcal{D}|y,
\end{equation}
we then finish the proof of remark.
  \end{proof}
Note that to prove Theorem \ref{main1}, it suffices to study the case where the center of rotation of $\mathcal{D}$ is the origin. From now on, unless specifically indicated, the rotating center of $\mathcal{D}$ will always assume to be the origin.
\subsection{Organization of the paper}
The rest of the paper will be organized as follows.  In Section 2, we establish the framework for the bifurcation analysis for the Rankine vortex and prove Theorem \ref{Gomez}.  The proof of some supporting facts will be left in the appendix. Then, in Section 3, we prove Theorem \ref{main1}. We will start by providing some ``coarse" estimates on the vortex patch via the relative stream function. In particular, we can prove the boundary of the patch can be parameterized by its augment. Then via analysis of the contour equation describing the V-states, we can prove the V-states are close to the Rankine vortex in some fine topology. As a consequence, Theorem \ref{Gomez} implies the desired symmetry property. In the appendix, we will collect and prove certain facts that we use throughout the paper. 
\subsection{Notations}
Throughout this paper, we reserve some characters for certain quantities according to the following rules:
\begin{itemize}
\item $C$, generic constants independent of the parameter $\delta$ which might change line to line.
    \item $\Delta$: the symmetric difference of two sets:
    \begin{equation}
        A\Delta B=(A-B)\cup (B-A). 
    \end{equation}
    \item $||$, in this paper, $||$ could mean the absolute value of a number, the area of a set, or the functional norm of a function.
   \item $\mathbb{D}$: the unit disk in $\mathbb{R}^2$.
   \item $\mathbb{T}$: the flat torus: the interval $[-\pi,\pi]$ with endpoints identified. 
   \item $\chi_{A}$, the characteristic function for $A$.
   \item $C_{center}^{1,\frac{1}{2}}:=\{f\in C^{1,\frac{1}{2}}(\mathbb{T}),f=\sum_{k\geq 1} a_k \sin{kx}+b_{k}\cos{kx}\}$  and the norm is the $C^{1,\frac{1}{2}}(\mathbb{T})$ norm.
   \item $C_{center}^{\frac{1}{2}}:=\{f\in C^{\frac{1}{2}}(\mathbb{T}),f=\sum_{k\geq 1} a_k \sin{kx}+b_{k}\cos{kx}$  and the norm is the sum of the $C^{\frac{1}{2}}(\mathbb{T})$ norm.  
   \item  $C_{even}^{1,\frac{1}{2}}:=\{f\in C^{1,\frac{1}{2}}(\mathbb{T}),f=\sum_{k\geq 1} a_{k}\cos{kx}\}$ and the norm is the $C^{1,\frac{1}{2}}(\mathbb{T})$ norm.
   \item $C_{odd}^{\frac{1}{2}}:=\{f\in C^{\frac{1}{2}}(\mathbb{T}),f=\sum_{k\geq 1} a_{k}\sin{kx}\}$ and the norm is the $C^{\frac{1}{2}}(\mathbb{T})$ norm.
\end{itemize}
\section{Bifurcation analysis near the Rankine vortex and the proof of Theorem \ref{Gomez}}
In this section, we perform a bifurcation analysis near the Rankine vortex and finish the proof of Theorem \ref{Gomez}.\par
In the context of Theorem \ref{Gomez}, the boundary of the patch is given by $R(\theta)$ where $\theta$ is the polar angle and we have $|R-1|_{C^{1,\frac{1}{2}}(\mathbb{T})}\leq \delta$. We will perform a bifurcation analysis for the contour equation in $C^{1,\frac{1}{2}}(\mathbb{T})$. Here, the analysis is greatly inspired by the framework of \cite{castro2016existence} and \cite{castro2016uniformly}.\par  
We define

$C_{center}^{1,\frac{1}{2}}:=\{f\in C^{1,\frac{1}{2}}(\mathbb{T}),f=\sum_{k\geq 1} a_k \sin{kx}+b_{k}\cos{kx}\}$  and the norm is the $C^{1,\frac{1}{2}}(\mathbb{T})$ norm, and $C_{center}^{\frac{1}{2}}:=\{f\in C^{\frac{1}{2}}(\mathbb{T}),f=\sum_{k\geq 1} a_k \sin{kx}+b_{k}\cos{kx}$  and the norm is the $C^{\frac{1}{2}}(\mathbb{T})$ norm.  
 Letting $\lambda:=\frac{1}{2\pi}\int_{0}^{2\pi}R(x)dx$, we have the following. \begin{equation}\label{stablity of first shell 2}
    |\lambda-1|\leq  C\delta.
\end{equation}
    Based on \eqref{stablity of first shell 2}, taking $R\rightarrow \frac{R}{\lambda}$, we can assume that $V:=R-1\in C_{center}^{1,\frac{1}{2}} $. 
    Correspondingly, we may parameterize the boundary of the rotating patch as \begin{equation}\label{contourGomez}
          \begin{bmatrix}
        &z_1(x,t)\\
        &z_2(x,t) 
    \end{bmatrix}= \begin{bmatrix}
        \cos{\Omega t} & \sin{\Omega t}  \\
        - \sin{\Omega t} & \cos{\Omega t} 
    \end{bmatrix} \begin{bmatrix}
        &(1+V(x))\cos{x}\\
        &(1+V(x))\sin{x}
    \end{bmatrix}.
    \end{equation}
    with $|V|_{C_{center}^{1,\frac{1}{2}}}\leq C\delta$.
    Now, with the form \eqref{contourGomez},
    we may write the contour equation of the rotating patch in terms of $V$ and $\Omega$.
    Now, we define \begin{equation}\label{F}
    \begin{aligned}
         &F_1(V):=\frac{1}{4 \pi}\int_{\mathbb{T}}\sin{(x-y)}\ln{\bigg((V(x)-V(y))^2+4(1+V(x))(1+V(y))\sin^2{\frac{x-y}{2}}\bigg)}\\&
         \bigg( (1+V(x))(1+V(y))+V^{'}(x)V^{'}(y)\bigg)dy,
    \end{aligned}
    \end{equation}
    \begin{equation}
    \begin{aligned}
         &F_2(V):=\frac{1+V(x)}{4 \pi}\int_{\mathbb{T}}\cos{(x-y)}\ln{\bigg((V(x)-V(y))^2+4(1+V(x))(1+V(y))\sin^2{\frac{x-y}{2}}\bigg)}\\&
         \bigg( V^{'}(y)-V^{'}(x)\bigg)dy,
    \end{aligned}
    \end{equation}
     \begin{equation}
    \begin{aligned}
         &F_3(V):=\frac{V^{'}(x)}{4 \pi}\int_{\mathbb{T}}\cos{(x-y)}\ln{\bigg((V(x)-V(y))^2+4(1+V(x))(1+V(y))\sin^2{\frac{x-y}{2}}\bigg)}\\&
         \bigg( V(x)-V(y)\bigg)dy,
    \end{aligned}
    \end{equation}
    and \begin{equation}\label{bifurcation}
            \mathcal{F}(\Omega,V):=\Omega(1+V(x)) V^{'}(x)-\sum_{i=1}^{3}F_{i}(V)(x).
    \end{equation}
    The contour equation for rotating patches now reads as \begin{equation}\label{contour}
        \mathcal{F}(\Omega,V)=0(\text{See \cite{castro2016existence} and \cite{castro2016uniformly}.}).
\end{equation}
A similar process as in \cite{castro2016existence} leads to the following.
\begin{lemma}\label{bifurcation1}
For the non-linear functional \eqref{bifurcation} we just defined, we have the following properties:
    \begin{itemize}
    \item[1.]
    $\mathcal{F}(\Omega,0)=0$ for every $\Omega$.
    \item [2.] 
    $\mathcal{F}(\Omega,V):\mathbb{R}\times V^{r}\rightarrow C_{center}^{\frac{1}{2}}$, where $V^{r}$ is an open neighborhood of 0 in $C_{center}^{1,\frac{1}{2}}$.
    \item[3.]
    The partial derivatives $\mathcal{F}_{\Omega}$, $\mathcal{F}_{V}$, and $\mathcal{F}_{\Omega V}$ exist and are continuous.
    \item [4.]  In the case where $\Omega=\frac{1}{4}$, define the linearized operator $\mathcal{L}$ as the linearized operator around the unit disk with the angular velocity $\frac{1}{4}$.
    We have \begin{equation}
        Ker(\mathcal{L})=Span\{(1,0),(0,\sin{2\theta}),(0,\cos{2\theta})\}.
    \end{equation}
\end{itemize} 
\end{lemma}
 We will give the proof of Lemma \ref{bifurcation1} in the appendix.
After finishing setting the bifurcation analysis, we are ready to prove Theorem \ref{Gomez}.
\begin{proof}[Proof of Theorem \ref{Gomez}]
Define $Kf:=\mathcal{F}_{V}f$ at the point $(\frac{1}{4},0)$.
We perform the Lyapunov Schmidt reduction as in \cite{kielhofer2012bifurcation} and write $C_{center}^{1,\frac{1}{2}}=Span\{\sin{2\theta},\cos{2\theta}\}\oplus Z$ and $C_{center}^{\frac{1}{2}}=Range K \oplus W$. 
For a uniformly rotating patch $(\Omega,V)$ near the disk, with $|\Omega-\frac{1}{4}|\leq \delta$ and $|V|_{C_{center}^{1,\frac{1}{2}}}\leq C \delta$, we have the following.\begin{equation}
    Q:=\mathbb{P}_{ker \mathcal{L}} (\Omega-\frac{1}{4},V)=(\Omega-\frac{1}{4},a \sin{2\theta}+b \cos{2\theta}),\text{ for some numbers $a$ and $b$.}
\end{equation} 
Since \eqref{contour} is rotational invariant in $V$, we may apply a suitable rotation so that   \begin{equation}
    \mathbb{P}_{ker \mathcal{L}} (\Omega-\frac{1}{4},V)=(\Omega-\frac{1}{4},c \cos{(2\theta)}),\text{for some real number $c$}.
\end{equation} 
Then, we consider the Lyapunov-Schmidt reduction map:
\begin{equation}
    Mf:=\mathbb{P}_{\text{Range}(K)}F \bigg((\Omega-\frac{1}{4},c \cos{2\theta})+(\frac{1}{4},f)\bigg). 
\end{equation}
From \eqref{contour}, we have  \begin{equation}\label{Bifurcation Paul}
    Mf=0.
\end{equation}
Then
M is a mapping from $V_{1}^{r}$ to $\text{Range} (K)$, where $V_{1}^{r}$ is a small neighborhood of $0$ in Z. Based on Lemma \ref{bifurcation1}, for  a sufficiently small given $(\Omega-\frac{1}{4},c\cos{2\theta})$, the linearized map of $M$ is bijective. Based on the contraction mapping theorem, for a given small enough $(\Omega-\frac{1}{4},c\cos{2\theta})$, there is a unique $f$ to \eqref{Bifurcation Paul} in $V_{1}^{r}$ independent of $(\Omega-\frac{1}{4},c \cos{2\theta})$(by shrinking $V_{1}^{r}$, if necessary). Furthermore, the solution $f$ that we just constructed would satisfy the requirement. \begin{equation}
    |f|_{C_{center}^{1,\frac{1}{2}}}\leq C\delta.
\end{equation}
We now define a ``even symmetric" subset of $C_{center}^{1,\frac{1}{2}}$:
    $C_{even}^{1,\frac{1}{2}}:=\{f\in C^{1,\frac{1}{2}}(\mathbb{T}),f=\sum_{k\geq 1} a_{k}\cos{kx}\}$ and the norm is the $C^{1,\frac{1}{2}}(\mathbb{T})$ norm. We also define the ``odd" symmetric subset of $C^{\frac{1}{2}}(\mathbb{T})$: $C_{odd}^{\frac{1}{2}}(\mathbb{T}):=\{f\in C^{\frac{1}{2}}(\mathbb{T}),f=\sum_{k\geq 1} a_{k}\sin{kx}\}$ and the norm is the $C^{\frac{1}{2}}(\mathbb{T})$ norm. 
    Due to the symmetry of $\mathcal{F}$, we can restrict the study of $\mathcal{F}$ in the regime from $\mathbb{R}\times \Tilde{V}^{r} $ to $C_{odd}^{\frac{1}{2}}$, where $\Tilde{V}^{r}$ is a neighborhood of $0$ in $C_{even}^{1,\frac{1}{2}}$. Again, by the same Lyapunov reduction trick, there is a unique $f \in \Tilde{V}_{r}$ to \eqref{Bifurcation Paul}, for any sufficiently small $(\Omega-\frac{1}{4},c\cos{2\theta})$. As a consequence, for the full equation without symmetry reduction, due to the special structure of $Ker \mathcal{L}$, by performing a suitable rotation, we may assume that $V \in C_{even}^{1,\frac{1}{2}}$.\par
\par Priorly, for $|\Omega-\frac{1}{4}|\leq C\delta$, there are solutions to \eqref{contour}, which are ellipses satisfying $|R(\theta)-1|_{C_{even}^{1,\frac{1}{2}}}\leq C\delta$.
   Moreover, as we restrict the study of \eqref{contour} in the space $\mathbb{R}\times C_{even}^{1,\frac{1}{2}}$, it can be verified that $\mathcal{F}$ near $(\frac{1}{4},0)$ satisfies the required conditions in the Crandall-Rabnowitz theorem (see Theorem \ref{kernerl1} in the appendix). Then the rigid perspectives of the Crandall-Rabnowitz theorem in \cite{crandall1971bifurcation} imply that the $C_{even}^{1,\frac{1}{2}}$ solutions mentioned above must be an ellipse or a disk. We then finish the proof.  
\end{proof}
\section{quantitative estimates regarding vortex patch}
In this section, we will prove the following theorem:\begin{theorem}\label{quan1}
     Let $\mathcal{D}$ be a uniformly rotating patch with angular velocity $\Omega$ with bounded simply-connected support. Assume that $\mathcal{D}$ satisfies  \begin{itemize}
        \item Angular velocity is close to $\frac{1}{4}$:
        $|\Omega-\frac{1}{4}|\leq \delta$; \item The geometric shape for the patch is close to the unit disk:
         $|\mathbb{D}\Delta \mathcal{D}|\leq \delta$.
\end{itemize}
Now, if $\delta$ is sufficiently small,
    then $\partial \mathcal{D}$ can be parameterized by its argument $\theta$: $R=R(\theta)$, with \begin{equation}
        |R-1|_{C^{1,\frac{1}{2}}(\mathbb{T})}\ll 1.
    \end{equation}
\end{theorem}
To prove Theorem \ref{quan1}, we first give some quantitative estimates of $\partial \mathcal{D}$ based on the relative stream function. 
In the setting of the vortex $\mathcal{D}$ rotating with angular velocity $\Omega$, there exists a relative stream function $\Psi$(see \cite{hassainia2020global}):
    \begin{equation}\label{relative stream function}
    \begin{aligned}
        &\Delta \Psi(x)=1-2\Omega, \text{for $x\in \mathcal{D}$},\\
        &\Delta \Psi(x)=-2\Omega, \text{for $x \in \mathcal{D}^{c}$},\\
        & \nabla (\Psi(x)+\frac{\Omega |x|^2}{2}) \rightarrow 0,\text{ as $ x\rightarrow \infty$,}\\ &\Psi(x)=0,\text{for $x \in \partial \mathcal{D}$.}      
    \end{aligned}
    \end{equation}
    Various results of rigidity regarding the vortex patch have come from the analysis of the relative stream function. (see \cite{fraenkel2000introduction}).
In this section, 
we will use the existence of the relative stream function for $\mathcal{D}$ to derive certain ``coarse" estimates of $\mathcal{D}$.
\begin{theorem}\label{Size of boundary}
    In the setting of Theorem \ref{quan1}, let $\delta\ll 1$, then for $x \in \partial \mathcal{D}$, \begin{equation}
        1-C\sqrt{\delta}\leq |x| \leq 1+C\sqrt{\delta}.
    \end{equation}
\end{theorem}
\begin{proof}

Due to the rotating vortex assumptions, we can define $\Psi$ as in \eqref{relative stream function}.
    In addition, we define the base relative stream function $\Psi_0$: 
   \begin{align*}
    \Psi_0(x):=&\left\{\begin{array}{cc}
         \frac{(1-2\Omega)(|x|^2-1)}{4}, &\quad 0<|x|<1,\\
         \frac{-\Omega (|x|^2-1)}{2}+\frac{1}{2} \ln{|x|}, &\quad |x|\geq 1.
    \end{array}\right.
\end{align*}
The $\Psi_0$ we just defined above is the relative stream function for the unit disk rotating with angular velocity $\Omega$.
Let $\Psi_1:=\Psi-\Psi_0$, we notice that \begin{equation}\label{stability}
    \Delta(\Psi_1)(x)=\left
    \{\begin{array}{cc}
         1, &\quad x\in \mathcal{D}-\mathbb{D};\\
         -1, &\quad x\in \mathbb{D}-\mathcal{D}.
    \end{array}\right.
\end{equation}
We now give a uniform upper bound for $\mathcal{D}$.
\begin{lemma}\label{Size of boundary1}
   We have a uniform upper bound for $\mathcal{D}$: $\mathcal{D}\subseteq B_{10000}$.  
\end{lemma}
\begin{proof}
    Let $R:=sup\{|x|,x\in \mathcal{D}\}$, we will prove $R<10000$.
    Now, for sufficiently small $\delta$, since $|\mathcal{D}\Delta \mathbb{D}|\leq \delta,$
     we can find a $x_0\in \partial \mathcal{D}$ such that \begin{equation}
        1-\sqrt{\delta}<|x_0|<1+\sqrt{\delta}.
    \end{equation}
    Based on the definition of $R$, there exists a $x_1 \in \partial \mathcal{D}$ such that \begin{equation}
        |x_1|=R.
    \end{equation}
    Since \begin{equation}
        \Psi(x_1)=\Psi(x_0)=0,
    \end{equation}
    we have \begin{equation}\label{error}
        |\nabla(\Psi_1)|_{L^\infty}|x_1-x_0|\geq |\Psi_1(x_1)-\Psi_1(x_0)|\geq \Psi_1(x_1)-\Psi_1(x_0)=\Psi_0(x_0)-\Psi_0(x_1).
    \end{equation}
    By the Steiner-type estimate(see Lemma \ref{Steiner} in the appendix), we have \begin{equation}
        |\nabla(\Psi_1)|_{L^\infty}\leq C\sqrt{\delta}.
    \end{equation}
    Thus, from \eqref{error} we have \begin{equation}
        C\sqrt{\delta}(R+1)\geq \frac{\Omega(R^2-1)}{2}-\frac{\ln{R}}{2}-100.
    \end{equation}
    Thus, for sufficiently small $\delta$, we have $R\leq 10000$.
\end{proof}
 Now, on the basis of Lemma \ref{Size of boundary1}, as a direct consequence of \eqref{error}, we have  \begin{equation}\label{stablity1}
    |\Psi_1(x)-\Psi_1(x_0)|\leq C\sqrt{\delta},\text{for all $x \in \partial \mathcal{D}$.} 
\end{equation}
 As a consequence, \begin{equation}
    |\Psi_0(x)-\Psi_0(x_0)|<C\sqrt{\delta},\text{for $x\in \partial \mathcal{D}$}.
\end{equation}
Using the explicit form of $\Psi_0$, we have the following. \begin{equation}
    1-C\sqrt{\delta}<|x|<1+C\sqrt{\delta},\text{for $x \in \partial \mathcal{D}.$} 
\end{equation}
We finish the proof of Theorem \ref{Size of boundary}.
\end{proof}
Based on the Theorem \ref{Size of boundary}, we can now finish the proof of Theorem \ref{quan1}:  
\begin{proof}[Proof of Theorem \ref{quan1}]
    Again by Lemma \ref{Steiner}, we have \begin{equation}\label{stablity2}
        |\nabla(\Psi-\Psi_0)|_{L^{\infty}}\leq C \sqrt{\delta}.
    \end{equation}
    Then by Theorem \ref{Size of boundary}, the explicit form of $\Psi_0$ and \eqref{stablity2}, for $\delta$ sufficiently small, we have the following control of the derivative of $\Psi$ in the polar coordinates:
    \begin{equation}\label{non-degeneracy}    
            |\Psi_{r}(x)-\frac{1}{4}|\leq C\sqrt{\delta}, |\Psi_{\theta}(x)|\leq C\sqrt{\delta},\text{for all $x \in B_{\frac{4}{3}}-B_{\frac{2}{3}}$. }
    \end{equation}
    Note $\partial \mathcal{D}$ is the level set of $\Psi$, and for $\delta$ sufficiently small, it should lie entirely in $B_{\frac{4}{3}}-B_{\frac{2}{3}}$. From \eqref{non-degeneracy}, by the inverse function theorem, $\mathcal{D}$ can be parameterized by its argument: $R=R(\theta)$, with \begin{equation}\label{contour12}
        R^{'}(\theta)=-\frac{\Psi_{\theta}(R(\theta)e^{i\theta})}{\Psi_{r}(R(\theta)e^{i\theta})}.
    \end{equation} 
    By \eqref{non-degeneracy}, we also have uniform control of  $C^{1}$ norm of $R$,\begin{equation} \label{R C1}
        |R-1|_{C^{1}(\mathbb{T})}\leq C\sqrt{\delta}. 
    \end{equation}
     By standard elliptic estimates and \eqref{non-degeneracy}, $\frac{\Psi_{\theta}}{\Psi_{r}}(x)$ has a uniform control in $C^{\frac{3}{4}}$ norm: \begin{equation}\label{classical1}
         |\frac{\Psi_{\theta}}{\Phi_{r}}|_{C^{\frac{3}{4}}\left(B_{\frac{4}{3}}-B_{\frac{2}{3}}\right)} \leq C.
     \end{equation} 
      \eqref{contour12}, \eqref{R C1}, and \eqref{classical1}  implies \begin{equation}
          |R|_{C^{1,\frac{3}{4}}(\mathbb{T})}\leq C.
      \end{equation}
      From interpolation theorem and again by \eqref{R C1},  we have \begin{equation}
          |R-1|_{C^{1,\frac{1}{2}}(\mathbb{T})}\ll 1,
      \end{equation}
      and we finish the proof.
\end{proof}
\section{Appendix}
We first recall the Steiner-type estimate in \cite{elgindi2022regular}.\begin{lemma}[Steiner type estimate]\label{Steiner} Let $\Omega\subset\mathbb{R}^2$ be bounded and measurable. Then, for every $(x,y)\in\mathbb{R}^2$, we have that
\[\int_{\Omega}\frac{1}{|(x,y)-(x_1,y_1)|}dx_1dy_1\leq C \sqrt{\big |\Omega\big |}.\]
\end{lemma}
For the rest of the appendix, we will give the proof of Lemma \ref{bifurcation1}.
\subsection{On the proof of Lemma \ref{bifurcation1}}
We will divide the proof of Lemma \ref{bifurcation1} into two parts, the first part would be to verify the disk is a solution to \eqref{contour} and the smooth property of the operator $\mathcal{F}$ we defined in \eqref{bifurcation}, namely the 1-3 properties in Lemma \ref{bifurcation1}. Then we get the kernel of the linearized operator $\mathcal{L}$ based on \cite{castro2016uniformly}.
\begin{proof}[Verification of 1-3 properties in Lemma \ref{bifurcation1}]
It is obvious that $\mathcal{F}(\Omega,0)=0$. Now we verify the smoothness of the mapping $\mathcal{F}$, and the $Frech\Acute{e}t$ differentiability of the mapping $\mathcal{F}$.
    We notice that we can decompose $F_1(V)$ into two parts: 
    \begin{equation}
    \begin{aligned}
        &\quad F_{1}(V)(x)\\&=-\frac{1}{4\pi}\int_{\mathbb{T}} \sin{(x-y)}\ln{\left(\dfrac{(V^{'})^2(x)+(1+V(x))(1+V(y))}{(1+V(x))(1+V(y))+\frac{(V(x)-V(y))^2}{4\sin^2{\frac{y-x}{2}}}}\right)}\\&\bigg( (1+V(x))(1+V(y))+V^{'}(x)V^{'}(y)\bigg)dy\\&+\frac{1}{4\pi}\int_{\mathbb{T}}\sin{(x-y)}\ln{(4\sin^2{\frac{x-y}{2}})} \ln{\left((V^{'})^2(x)+(1+V(x))(1+V(y))\right)}\\&\bigg( (1+V(x))(1+V(y))+V^{'}(x)V^{'}(y)\bigg)dy\\&         =\int_{\mathbb{T}} sin(x-y)K_{11}(V)(x,y)dy +\int_{\mathbb{T}} sin(x-y)\ln{(4sin^{2}{\frac{x-y} 2})} K_{12}(V)(x,y)dy
        \\&:=F_{11}(V)(x)+F_{12}(V)(x).
    \end{aligned}
    \end{equation}
    Due to the fact that $K_{11}(V)$ and $K_{12}(V)$ are smooth maps that map $V^{r}$ to $C^{\frac{1}{2}}(\mathbb{T})\times C^{\frac{1}{2}}(\mathbb{T})$, the regularity of $\sin{(x-y)}$ and $\sin{(x-y)}\ln(4 sin^{2}{\frac{y-x}{2}})$ implies that $F_{1}(V)$ is in $C^{\frac{1}{2}}(\mathbb{T})$ and $F_{1}(V)$ is twice $Frech\Acute{e}t$ differentiable in $V$ from $C_{center}^{1,\frac{1}{2}}$ to $C^{\frac{1}{2}}(\mathbb{T})$ (The interested readers may refer to \cite{castro2016uniformly} for more details).
    We may apply the same methods to $F_2(V)$ and $F_3(V)$ and derive that $F_2$ and $F_3$ are twice $Frech\Acute{e}t$ differentiable in $V$ from $C_{center}^{1,\frac{1}{2}}$ to $C^{\frac{1}{2}}(\mathbb{T})$. As a consequence, we have  $\mathcal{F}$ maps $\mathbb{R}\times V^{r}$ to $C_{center}^{1,\frac{1}{2}}$, and $\mathcal{F}_{\Omega}$, $\mathcal{F}_{V}$, $\mathcal{F}_{\Omega V}$ exist and are continuous. 
\end{proof}
The most tricky part in the proof of Lemma \ref{bifurcation1} is to calculate the kernel of the linearized operator $\mathcal{L}$. A statement concerning the kernel of $\mathcal{L}$ has been given in the proof of Theorem 1.2 in \cite{castro2016uniformly}:
\begin{theorem}\label{kernerl1}Let $C_{odd}^{\frac{1}{2}}:=\{f\in C^{\frac{1}{2}}(\mathbb{T}),f=\sum_{k\geq 1} a_{k}\sin{kx}\}$ and the norm is the $C^{\frac{1}{2}}(\mathbb{T})$ norm.  We now
restrict the study of $\mathcal{F}$ in $\mathbb{R}\times C_{even}^{1,\frac{1}{2}}$, and we can prove the following:
\begin{itemize}
     \item[1.]
    $\mathcal{F}(\Omega,0)=0$ for every $\Omega$.
    \item [2.] 
    $\mathcal{F}(\Omega,V):\mathbb{R}\times V_{1}^{r}\rightarrow C_{odd}^{\frac{1}{2}}$, where $V_{1}^{r}$ is an open neighborhood of 0 in $C_{even}^{1,\frac{1}{2}}$.
    \item[3.]
    The partial derivatives $\mathcal{F}_{\Omega}$, $\mathcal{F}_{V}$, and $\mathcal{F}_{V\Omega}$ exist and are continuous.
    \item [4.] 
   Let $\mathcal{K}f:=\mathcal{F}_{V}f$ at the point $(\frac{1}{4},0)$. $\text{ker} \ \mathcal{K}$ and $C_{odd}^{\frac{1}{2}}/Range(\mathcal{K})$ are one-dimensional at $V=0$ and $\Omega=\frac{1}{4}$, and $\text{Ker}\  \mathcal{K}=\text{span}\{\cos{2\theta}\}$.
    \item[5] $\mathcal{F}_{\Omega V}((0,\cos{2\theta}))\notin \text{Range}(\mathcal{K})$ at $(\frac{1}{4},0)$.
\end{itemize} 
\end{theorem}
In \cite{castro2016uniformly}, the authors discuss the proof of Theorem \ref{kernerl1} in much more general cases including g-Sqg. Here, for completeness, we will give the verification of Theorem \ref{kernerl1} under the setting of 2D Euler. 
\begin{proof}[Verification of Theorem \ref{kernerl1}]
Recall $\mathcal{F}$ is defined in \eqref{bifurcation}. By the symmetry of the kernel, it is easy to verify $F$ maps $\mathbb{R}\times V_{1}^{r}$ to $C_{odd}^{\frac{1}{2}}$. Similar to what we did in the proof of Lemma \ref{bifurcation1}, we can prove various smooth properties of $\mathcal{F}$. We now proceed to make explicit calculations to finish the rest of the proof. We have the following. \begin{equation}\label{first term}
        \begin{aligned}
            &\quad \mathcal{F}_{1V} (f)\\
            &=\frac{1}{4\pi}\int_{\mathbb{T}}\sin{(x-y)}[f(x)+f(y)]+\sin{(x-y)}\ln{\left(4\sin^2{\left(\frac{y-x}{2}\right)}\right)}[f(x)+f(y)]dy\\
            &=\frac{1}{4\pi}\int_{\mathbb{T}}\sin{(x-y)}f(y)+\sin{(x-y)}\ln{\left(4\sin^2{\left(\frac{y-x}{2}\right)}\right)}f(y)dy.
        \end{aligned}
    \end{equation}
    Similar calculation leads to
    \begin{equation}\label{second term}
    \begin{aligned}
        &\quad \mathcal{F}_{2V}(f)\\&=\frac{1}{4\pi}\int_{\mathbb{T}}cos{(x-y)} \ln{\left(4\sin^2{(\frac{x-y}{2})}\right)}[f^{'}(y)-f^{'}(x)]dy\\&=\frac{1}{4\pi}\int_{\mathbb{T}}cos{(x-y)} \ln{\left(4\sin^2{(\frac{x-y}{2})}\right)}[f^{'}(y)-f^{'}(x)]dy; 
    \end{aligned}
    \end{equation}
    and \begin{equation}\label{third term}
     \mathcal{F}_{3V}(f)=0.
\end{equation}
Now based on \eqref{first term},\eqref{second term} and \eqref{third term}, we have 
\begin{equation}
\begin{aligned}
    &\quad \mathcal{F}_{V} f
    =\frac{-1}{4\pi}\int_{\mathbb{T}}cos{(x-y)} \ln{\sin^2{(\frac{x-y}{2})}}dy f^{'}(x)\\&+\frac{1}{4\pi}\int_{\mathbb{T}} \ln{\left(4\sin^2{(\frac{x-y}{2})}\right)}[f^{'}(y)\cos{(x-y)}+f(y)\sin{(x-y)}]+\sin{(x-y)}f(y)dy
    \\&:=I_1+I_2.
\end{aligned}
\end{equation}
Estimates on $I_1$:
Based on integration by parts, we have 
\begin{equation}\label{I1}
\begin{aligned}
    &\quad\int_{\mathbb{T}} \cos{(x-y)} \ln{\sin^2{(\frac{x-y}{2})}}dy\\&=\int_{\mathbb{T}}D_{y}(-\sin{(x-y)}) \ln{\left(1-\cos{(x-y)}\right)}dy\\&= \int_{\mathbb{T}} D_{y}\left(-\sin{(x-y)}\ln{\left(1-\cos{(y-x)}\right)}\right)dy\\&+\int_{\mathbb{T}}\sin{(x-y)}\frac{\sin{(y-x)}}{1-\cos{(y-x)}}dy\\&=\int_{\mathbb{T}}1+\cos{(y-x)}dy\\&=2\pi.
\end{aligned}
\end{equation}
As a direct consequence of \eqref{I1}, we have \begin{equation}\label{I_1}
    I_1=-\frac{1}{2}f^{'}(x).
\end{equation}
Estimates on $I_2$:
\begin{equation}\label{I2}
\begin{aligned}
    &\quad\int_{\mathbb{T}} \ln{\left(4\sin^2{\frac{x-y}{2}}\right)}[\cos{(x-y)}f^{'}(y)+\sin{(x-y)}f(y)]dy\\&=P.V \int_{\mathbb{T}} \ln{\left(4\sin^2{(\frac{y-x}{2})}\right)} D_{y}[\cos{(x-y)}f(y)]dy\\&=P.V \int_{\mathbb{T}}D_{y}\left(4\ln{\left(\sin^2{(\frac{y-x}{2})}\cos{(y-x)}\right)}f(y)\right)dy-P.V\int_{\mathbb{T}}f(y)\frac{\cos{(\frac{y-x}{2})}\cos{(y-x)}}{\sin{\frac{y-x}{2}}}dy\\&=:
    J_1+J_2.
\end{aligned}
\end{equation}
By the Holder regularity of $f$, we have \begin{equation}\label{J1}
    J_1=0.
\end{equation}
As a result of \eqref{I2} and \eqref{J1}, we have \begin{equation}\label{I_2}
\begin{aligned}
    &\quad I_2=\frac{-1}{4\pi}P.V\int_{\mathbb{T}}f(y)\frac{\cos{(\frac{y-x}{2}})}{\sin{(\frac{y-x}{2}})}dy\\&
    =\frac{-1}{4\pi}\int_{\mathbb{T}}(f(y)-f(x))\frac{\cos{(\frac{y-x}{2}})}{\sin{(\frac{y-x}{2})}}dy.
\end{aligned}
\end{equation} 
Now, use the ansatz \begin{equation}\label{ansatz}
       f(x)=\sum_{k=1}^{\infty}a_k \cos{(kx)}=\sum_{k}a_k \frac{e^{ikx}+e^{-ikx}}{2},
\end{equation}
from \eqref{I_2}, we have \begin{equation}\label{I21}
    \begin{aligned}
        &\quad I_2= \frac{i}{4\pi} \sum_{k=1}^{\infty} \int_{\mathbb{T}} a_{k} \frac{e^{ikx} + e^{-ikx} - e^{iky} - e^{-iky}}{2} 
        \frac{e^{\frac{i(y-x)}{2}} + e^{\frac{-i(y-x)}{2}}}{e^{\frac{i(y-x)}{2}}- e^{\frac{-i(y-x)}{2}}} dy \\
        & = \frac{i}{4\pi} \sum_{k=1}^{\infty} \int_{\mathbb{T}} a_{k} \frac{e^{ikx} + e^{-ikx} - e^{iky} - e^{-iky}}{2} 
        (2 \sum_{j=0}^{\infty} e^{-ij(y-x)} - 1) dy\\&=\sum_{k}a_{k}\sin{kx}.
    \end{aligned}
\end{equation}
Now, based on \eqref{I_1} and \eqref{I21},(again under \eqref{ansatz}) we have \begin{equation}\label{Range}
\mathcal{K}f=\sum_{k=1}^{\infty}(1-\frac{k}{2})a_{k}\sin{k\theta}.
\end{equation}
From \eqref{Range}, we also get $C_{odd}^{\frac{1}{2}}/\text{Range $(\mathcal{K})$}=\text{span}\{\sin{2\theta}\}$.
Thus, $\mathcal{F}_{\Omega V}\left((0,\cos{2\theta})\right)=-2\sin{2\theta}\notin \text{Range}(\mathcal{K})$. 
We now finish the proof of Theorem \ref{kernerl1}.
\end{proof}
Now that we have verified Theorem \ref{kernerl1}, we are ready to finish the proof of Lemma \ref{bifurcation1}.
\begin{proof}[On the Calculations of the Kernel of $\mathcal{L}$]
For a general $f\in C_{center}^{1,\frac{1}{2}}$, we may write \begin{equation}
    f(x)=\sum_{k=1}^{\infty}a_k\cos{k\theta}+b_{k}\sin{k\theta},
\end{equation}
similar calculation in the proof of Theorem \ref{kernerl1} leads to \begin{equation}
    \mathcal{F}_{V}f=\sum_{k=1}^{\infty}(1-\frac{k}{2})a_k\sin{k\theta}-(1-\frac{k}{2})b_{k}\cos{k\theta}.
\end{equation}
Then it is obvious that the kernel of $\mathcal{L}=Span\{(1,0),(0,\sin{2\theta}),(0,\cos{2\theta})\}$.
\end{proof}
\section*{Acknowledgement}
The motivation of this project was uncovered in an enlightening conversation with Xiaoyutao Luo. The author thanks Siming He for providing various suggestions in the writing of the article. Moreover, the author thanks Jamin Park and Yao Yao for inspiring discussions. The author also thanks his advisor Tarek Elgindi for constant support and encouragement. The project is partially supported by the NSF DMS 2304392. 
\bibliographystyle{plain}
\bibliography{main}
\begin{itemize}
    \item Department of mathematics, Duke University, Durham, NC\\
    Email address: yh298@duke.edu
\end{itemize}
\end{document}